\documentclass[a4paper,11pt]{article}

\usepackage{amsmath,amsthm,amsfonts,amssymb, color}
\usepackage{graphicx}
\usepackage [english]{babel}
\usepackage[latin1]{inputenc}
\usepackage[all]{xy}
\usepackage{enumerate}

\newtheorem{theorem}{Theorem}[section]
\newtheorem{lem}[theorem]{Lemma}
\newtheorem{cor}[theorem]{Corollary}
\newtheorem{prop}[theorem]{Proposition}
\newtheorem{thm}[theorem]{Theorem}
\newtheorem{ex}[theorem]{Example}

\theoremstyle{definition}
\newtheorem{defn}[theorem]{Definition}
\newtheorem{rem}[theorem]{Remark}

\newcommand{\cd}{\cdot}
\newcommand{\ra}{\rightarrow}
\newcommand{\Ra}{\Rightarrow}
\newcommand{\RLa}{\Leftrightarrow}
\newcommand{\rla}{\leftrightarrow}
\newcommand{\s}{\sim}
\newcommand{\we}{\wedge}
\newcommand{\ES}{\mathrm{ES}}
\newcommand{\IS}{\mathrm{IS}}
\newcommand{\WIS}{\mathrm{hIS}}
\newcommand{\WISWf}{\mathrm{hIS_{\mathrm{4}}}}
\newcommand{\WISWfi}{\mathrm{hIS_{\mathrm{5}}}}

\newcommand{\gEA}{\mathrm{ShIS}}
\newcommand{\gEAE}{\mathrm{ShIS_{E}}}

\begin{document}

\title{On a symmetrization of hemiimplicative semilattices}

\author{José Luis Castiglioni and Hernán Javier San Martín}

\date{}
\maketitle

\begin{abstract}
A hemiimplicative semilattice is a bounded semilattice $(A,
\wedge, 1)$ endowed with a binary operation $\ra$, satisfying that
for every $a, b, c \in A$, $a \leq b \ra c$  implies  $a \we b
\leq c$ (that is to say, one of the conditionals satisfied by the
residuum of the infimum) and the equation $a \ra a = 1$. The class
of hemiimplicative semilattices form a variety.  These structures
provide a general framework for the study of different structures
of interest in algebraic logic.
In any hemiimplicative semilattice it is possible to define a
derived operation by $a \s b := (a \ra b) \wedge (b \ra a)$.
Endowing $(A, \wedge, 1)$ with the binary operation $\s$ results
again a hemiimplicative semilattice, which also satisfies the
identity $a \s b = b \s a$. We call the elements of the subvariety
of hemiimplicative semilattices satisfying $a \ra b = b \ra a$, a
symmetric hemiimplicative semilattice. In this article, we study the
correspondence assigning the symmetric hemiimplicative
semilattice $(A, \wedge, \s , 1)$ to the hemiimplicative semilattice
$(A, \wedge, \ra, 1)$. In particular,
we characterize the image of this correspondence. We also provide many new
examples of hemiimplicative semilattice structures on any bounded
semillatice (possibly with bottom). Finally, we characterize
congruences on the clases of hemiimplicative semilattices
introduced as examples and we describe the principal congruences
of hemiimplicative semilattices.
\end{abstract}

\section{Introduction}
\label{int}

Recall that a structure $(A,\leq, \cd, e)$ is said to be a
partially ordered monoid if $(A,\leq)$ is a poset, $(A,\cd, e)$ is
a monoid and for all $a,b,c\in A$, if $a \leq b$ then $a\cd c \leq
b \cd c$ and $c\cd a \leq c \cd b$. Although commutativity do not
play any special feature in the discussion that follow, we shall
assume in this article that all monoids are commutative in order
to make the exposition more clear.

Let us also recall that the residuum (when it exists) of the
monoid operation of a partially ordered monoid $(A,\leq, \cd, e)$
is a binary operation $\ra$ on $A$ such that for every $a$, $b$
and $c$ in $A$,
\begin{equation*}\label{residuum}
a \cd b \leq c \textrm{ if and only if } a \leq b \ra c.
\end{equation*}

Note that previous equivalence can be seen as the conjunction of
the following conditionals:

\begin{enumerate}
\item[(r)] If  $a \cd b \leq c$  then  $a \leq b \ra c$ and
\item[(l)] If  $a \leq b \ra c$  then  $a \cd b \leq c$.
\end{enumerate}

This suggest us to consider binary operations $\ra$ satisfying
either (r) or (l) above. We call such operations r-hemiresidua
(l-hemiresidua) of the monoid operation respectively.

Let $\mathbf{A} = (A,\leq, \cd, e)$ be a partially ordered monoid.
An \emph{r-hemiresiduated monoid} is the expansion of $\mathbf{A}$
with an r-hemiresiduum. Similarly, we define an
\emph{l-hemiresiduated monoid}.
Clearly, every partially ordered residuated monoid is both an
r-hemiresiduated monoid and an l-hemiresiduated monoid.
Straightforward calculations show that the $BCK$-algebras with
meet \cite{Id84a,Kuhr} are examples of r-hemiresiduated monoids
which in general are not residuated. Some examples of
l-hemiresiduated monoids which in general are not residuated can
be found for instance in \cite{SM1}.

\begin{rem}\label{mp}
Let $(A,\leq, \cd, e)$ be a partially ordered monoid and $\ra$ a
binary operation on $A$. Then $(A,\leq, \cd, \ra, e)$ is an
l-hemiresiduated monoid if and only if $a\cd (a\ra b) \leq b$ for
every $a,b \in A$. In order to prove it, suppose that $(A,\leq,
\cd, \ra, e)$ is an l-hemiresiduated monoid. For every $a,b\in A$
we have that $a\ra b\leq a\ra b$, so $a\cd (a\ra b) \leq b$.
Conversely, suppose that $a\cd (a\ra b) \leq b$ for any $a,b\in
A$. Let $a,b,c\in A$ such that $a\leq b \ra c$. Then $a\cd b \leq
b\cd (b\ra c) \leq c$, so $a\cd b \leq c$. Therefore, $(A,\leq,
\cd, \ra, e)$ is an l-hemiresiduated monoid.
\end{rem}

In this work we shall be particularly interested in the case that
the partially ordered monoid is idempotent; more precisely, a meet
semillatice whose order agrees with that induced by the infimum.
We call \emph{r-hemiresiduated semilattices}
(\emph{l-hemiresiduated semilattices}) to the expansion of these
structures with a r-hemiresiduum (l-hemiresiduum). Throughout this
paper we write semilattice in place of meet semilattice. A
semilattice $(A,\we)$ is said to be bounded if it has a greatest
element, which will be denoted by $1$; in this case we write
$(A,\we,1)$. A \emph{l-hemiimplicative semilattice} is an
l-hemiresiduated semilattice satisfying some additional
conditions. Since we shall only consider l-hemi-implicative
semilattice, in what follows we shall omit the l- prefix.

\begin{defn}
A \emph{hemiimplicative semilattice} is an algebra $(A,\we,\ra,1)$
of type $(2,2,0)$ which satisfies the following conditions:
\begin{enumerate}
\item[(H1)] $(A,\we,1)$ is a bounded semilattice, \item[(H2)] for
every $a \in A$, $a \ra a = 1$ and \item[(H3)] for every $a,b,c
\in A$, if $a\leq b\ra c$ then $a\we b \leq c$.
\end{enumerate}
\end{defn}

Notice that (H3) is the condition (l) for the case in which $\cd =
\we$. Hemiimplicative semilattices were called weak implicative
semilattices in \cite{SM2}. Let $(A,\we)$ be a semilattice and
$\ra$ a binary operation. Then by Remark \ref{mp} we have that $A$
satisfies (H3) if and only if for every $a,b \in A$ the equation
$a\we (a\ra b)\leq b$ is satisfied. In every hemiimplicative
semilattice we define $a \rla b: = (a \ra b) \we (b \ra a)$. We
write $\WIS$ for the variety of hemiimplicative semilattices.

\begin{rem}
Let $A\in \WIS$ and $a,b\in A$. Then $a = b$ if and only if $a\rla
b = 1$. We also have that $1\ra a\leq a$.
\end{rem}

Recall that an implicative semilattice \cite{Cu,N} is an algebra
$(A, \we, \ra)$ of type $(2,2)$ such that $(A,\we)$ is
semilattice, and for every $a,b,c\in H$ it holds that $a\we b \leq
c$ if and only if $a\leq b\ra c$. Every implicative semilattice
has a greatest element. In this paper we shall include the
greatest element in the language of the algebras. We write $\IS$
for the variety of implicative semilattices.

\begin{rem}
Clearly, $\IS$ is a subvariety of $\WIS$. Other examples of
hemiimplicative semilattices are the $\{\we,\ra, 1\}$-reduct of
semi-Heyting algebras \cite{S1,S2} and the $\{\we, \ra,
1\}$-reduct of some algebras with implication \cite{CCSM}, as for
example the $\{\we,\ra,1\}$-reduct of RWH-algebras \cite{CJ}.
\end{rem}

In \cite{Jen} Jenei shows that the class of BCK algebras with meet
is term equivalent to the class of equivalential equality
algebras, and he defines the equivalence operation $\s$ in terms
of the implication in the usual way; i.e., $a \s b := a \rla b$.
Some of these ideas were generalized and studied for pseudo
$BCK$-algebras \cite{Ciu,D,Jen2}.

In particular, the variety of implicative semilattices is term
equivalent to a subvariety of that of equivalential equality
algebras. Let us write $\ES$ for this subvariety. The algebras in
the class $\ES$ satisfy:
\begin{enumerate}
\item[a)] $a \s b = b \s a$, \item[b)] $a \s a = 1$, \item[c)] $a
\we (a \s b) \leq b$.
\end{enumerate}

On the other hand, implicative semilattices satisfy b) and c)
above, but of course not necessarily a). Hence there seems to be a
common frame for both classes of algebras, where the algebras in
$\ES$ may be seem as elements with a symmetric implication and the
construction $a \s b := a \rla b$ a sort of symmetrization of the
original implication. In this paper we explore a convenient
framework where the aforementioned intuitions could be made
precis. Most results concerning the relation between implicative
semilattices and the class $\ES$ are part of the folklore.
However, for the sake of completeness, we shall recall some basic
results in Section \ref{basic}.

In Section \ref{section3} some subvarieties of hemiimplicative
semilattices are presented. New examples of hemiimplicative
semilattices are provided, by defining a suitable structure on any
bounded (sometimes with bottom) semilattices. The relationship
between the variety of hemiimplicative semilattices and its
subvariety of symmetric elements is studied.

In Section \ref{section4} we characterize congruences on the
clases of hemiimplicative semilattices introduced in the examples
of Section \ref{section3} and we describe the principal
congruences of hemiimplicative semilattices.

\section{Relation between $\IS$ and $\ES$} \label{basic}

As we have mentioned before, the relation between implicative
semilattices and the class $\ES$ is part of the folklore. However,
for the sake of completeness, we shall make explicit some details
about this relation.

\begin{rem} \label{r1}
\begin{enumerate}[\normalfont a)]
\item In every implicative semilattice $(A,\we,\ra,1)$ we have
that $(A,\we,1)$ is a bounded semilattice and $a\rla a = 1$ for
every $a\in A$. We also have that for every $a,b,c\in A$, $c \leq
a \rla b$ if and only if $a\we c = b \we c$. \item Implicative
semilattices satisfy $a\ra b = a \rla (a\we b)$ for every $a,b\in
A$. \item Consider an algebra $(A,\we,\s,1)$ of type $(2,2,0)$
such that $(A,\we)$ is a semilattice. For every $a,b,c\in A$ we
consider the following conditions: 1) $a \we (a\s b) = b \we (a\s
b)$ and 2) if $a\we c = b\we c$ then $c \leq a \s b$. For every
$a,b,c \in A$ conditions 1) and 2) are satisfied if and only if we
have that $a\s b =$ max $\{c \in A: a\we c = b\we c\}$ for every
$a,b\in A$.
\end{enumerate}
\end{rem}

In the following proposition we consider a particular class of
algebras.

\begin{prop} \label{isl1}
Let $(A,\we,\s,1)$ be an algebra of type $(2,2,0)$ such that
satisfies the following conditions:
\begin{enumerate}[\normalfont 1)]
\item $(A,\we,1)$ is a bounded semilattice, \item $a\s a = 1$,
\item $a \we (a \s b) = b \we (a\s b)$, \item if $a\we c = b\we c$
then $c \leq a \s b$.
\end{enumerate}
Then $(A,\we, \Ra,1)\in \IS$, where $\Ra$ is defined by $a\Ra b =
a \s (a\we b)$. Moreover, we can replace the condition $4)$ by the
inequality
\begin{enumerate}[\normalfont 4')]
\item $c\we ((a\we c) \s (b\we c)) \leq a\s b$.
\end{enumerate}
Therefore, the class of algebras of type $(2,2,0)$ which satisfy
the conditions $1)$, $2)$, $3)$ and $4)$ is a variety.
\end{prop}

\begin{proof}
In order to prove that $(A,\we,\Ra,1)\in \IS$, we only need to
prove that for every $a,b,c\in A$, $a\leq b \Ra c$ if and only if
$a\we b \leq c$. Suppose that $a\leq b\Ra c$, i.e., $a\leq b \s
(b\we c)$. Then $a\we b \leq b \we (b \s (b\we c))$. It follows
from 3) that $b \we (b \s (b\we c)) = (b\we c) \we (b\s (b\we c))
\leq c$, so $a\we b \leq c$. Conversely, suppose that $a\we b \leq
c$. Then $a\we b = a \we (b\we c)$. Taking into account 4) we have
that $a\leq b \s (b\we c)$, i.e., $a\leq b\Ra c$.

Finally we will prove the equivalence between 4) and 4'). Assume
the condition 4). Since $a\we (c \we ((a\we c) \s (b\we c))) =
b\we (c \we ((a\we c) \s (b\we c)))$ then $c\we ((a\we c) \s (b\s
c)) \leq a\s b$, which is the condition 4'). Conversely, assume
the condition 4'), and suppose that $a\we c = b\we c$. It follows
from the conditions 2) and 4') that $c = c\we ((a\we c) \s (b\we
c)) \leq a\s b$, so $c\leq a \s b$.
\end{proof}

We write $\ES$ for the variety of algebras of type $(2,2,0)$ which
satisfy the conditions $1)$, $2)$, $3)$ and $4)$ of Proposition
\ref{isl1}. The following corollary follows from Proposition
\ref{isl1} and Remark \ref{r1}.

\begin{cor} \label{isl2}
\begin{enumerate} [\normalfont 1)]
\item If $(A,\we, \ra,1) \in \IS$ then $(A,\we,\rla,1) \in \ES$.
Moreover, for every $a,b\in A$ we have that $a\ra b = a\Ra b$,
where $\Ra$ is the implication associated to the algebra
$(A,\we,\rla,1)$. \item If $(A,\we,\s,1)\in \ES$ then
$(A,\we,\Ra,1) \in \IS$. Moreover, for every $a,b\in A$ we have
that $a\s b = a\Leftrightarrow b$. \item The varieties $\IS$ and
$\ES$ are term equivalent.
\end{enumerate}
\end{cor}


\section{Hemiimplicative semilattices and symmetric
hemiimplicative semilattices} \label{section3} In this section we
study the variety of hemiimplicative semilattices and some of its
subvarieties. In particular, we present some general examples by
defining hemiimplicative structures on any bounded semilattice. In
a second part we introduce a new variety, whose algebras will be
called for us \emph{symmetric hemiimplicative semilattices}. The
original motivation to consider this variety follows from the
properties of the algebras $(A,\we, \rla, 1)$ associated to the
algebras $(A,\we,\ra,1) \in \WIS$. \vspace{1pt}

There are several ways of defining a hemiimplicative structure on
any bounded semilattice. Some of this ways are described in the
examples below. Note that some of this procedures only apply to
bounded semilattices with bottom.

\begin{ex}\label{ejs1}
Let $(A, \we, 1)$ be a bounded semilattice (with bottom $0$, when
necessary). We define binary operations $\ra$ on $A$ that makes
the algebra $(A, \we, \ra, 1)$ a hemiimplicative semilattice.
\begin{equation}
\label{ej3} a \ra b = \left\{
\begin{array}{c}
1 \hskip10pt \textrm{if $a = b$}\\
0  \hskip10pt \textrm{if $a\neq b$}\\
\end{array}
\right.
\end{equation}

\begin{equation}
\label{ej6} a \ra b = \left\{
\begin{array}{c}
1 \hskip10pt \textrm{if $a \leq b$}\\
b  \hskip10pt \textrm{if $a\nleq b$}\\
\end{array}
\right.
\end{equation}
\end{ex}

In the context of algebras of $\WIS$ consider the following two
equations:
\begin{enumerate}
\item[$\mathrm{(H4)}$] $a\ra (a\we b) = a \ra b$,
\item[$\mathrm{(H5)}$] $(a \we b) \ra b = 1$.
\end{enumerate}

We write $\WISWf$ for the subvariety of $\WIS$ whose algebras
satisfy $\mathrm{(H4)}$, and $\WISWfi$ for the subvariety of
$\WIS$ whose algebras satisfy $\mathrm{(H5)}$.

\begin{rem}
Let $A\in \WISWf$ and $a\in A$. Then $a\ra 1 = 1$. It follows from
that $a\ra 1 = a \ra (a\we 1)$ and $a\ra a = 1$.
\end{rem}

\begin{prop} \label{ps}
The following inclusions of varieties are proper: $\WISWf
\subseteq \WISWfi \subseteq \WIS$.
\end{prop}

\begin{proof}
In order to prove that $\WISWf$ is a subvariety of $\WISWfi$, let
$A\in \WISWf$ and let $a,b\in A$. Then $(a\we b) \ra b = (a\we b)
\ra ((a\we b)\we b)$. But $(a\we b)\we b = a\we b$ and $(a\we b)
\ra (a\we b) = 1$. Then $(a\ra b)\ra b = 1$, so $A \in \WISWfi$.
In order to show that $\WISWf$ is a proper subvariety of
$\WISWfi$, consider the boolean lattice $B_4$ of four elements,
where $x$ and $y$ are the atoms, and consider the implication
given in (\ref{ej6}) of Example \ref{ejs1}. Then $(a\we b) \ra b =
1$ for every $a,b$. However, $x\ra (x\we y) = x\ra 0 = 0$ and
$x\ra y = y$, so $x\ra (x\we y) \neq x\ra y$. Thus, $\WISWf$ is a
proper subvariety of $\WISWfi$.

Finally we shall show that $\WISWfi$ is a proper subvariety of
$\WIS$. Consider $B_4$ with the implication given in (\ref{ej3})
of Example \ref{ejs1}. Then $(x\we y) \ra y = 0$. Therefore, the
equation $\mathrm{(H5)}$ is not satisfied.
\end{proof}

Let $A\in \WIS$ and $a,b\in A$. Notice that if $a\ra b = 1$ then
$a\leq b$ because $a = a\we 1 = a\we (a\ra b) \leq b$. In the
following corollary we characterize the hemiimplicative
semilattices in which holds the converse property, that is to say,
the condition if $a\leq b$ then $a\ra b = 1$.

\begin{cor}
Let $A\in \WIS$. The following conditions are equivalent:
\begin{enumerate}[\normalfont 1)]
\item $A\in \WISWfi$. \item For every $a,b\in A$, $a\leq b$ if and
only if $a\ra b = 1$.
\end{enumerate}
\end{cor}

\begin{proof}
Suppose that $A\in \WISWfi$. Let $a\leq b$. Then $1 = (a\we b) \ra
b = a\ra b$. Conversely, suppose that for every $a,b\in A$, $a\leq
b$ if and only if $a\ra b = 1$. Since $a\we b \leq b$ then $(a\we
b)\ra b = 1$. Therefore, $A\in \WISWfi$.
\end{proof}

\begin{ex} \label{ejs2}
Let $(A, \we, 1)$ be a bounded semilattice (with bottom $0$, when
necessary). We define binary operations $\ra$ on $A$ that makes
the algebra $(A, \we, \ra, 1)$ a hemiimplicative semilattice.

\begin{equation}
\label{ej1} a \ra b = \left\{
\begin{array}{c}
1 \hskip10pt \textrm{if $a = b$}\\
b  \hskip10pt \textrm{if $a\neq b$}\\
\end{array}
\right.
\end{equation}

\begin{equation}
\label{ej2} a \ra b = \left\{
\begin{array}{c}
1 \hskip28pt \textrm{if $a \leq b$}\\
a\we b  \hskip10pt \textrm{if $a\nleq b$}\\
\end{array}
\right.
\end{equation}

\begin{equation}
\label{ej4} a \ra b = \left\{
\begin{array}{c}
1 \hskip28pt \textrm{if $a = b$}\\
a\we b  \hskip10pt \textrm{if $a\neq b$}\\
\end{array}
\right.
\end{equation}

\begin{equation}
\label{ej5} a \ra b = \left\{
\begin{array}{c}
1 \hskip10pt \textrm{if $a \leq b$}\\
0  \hskip10pt \textrm{if $a\nleq b$}\\
\end{array}
\right.
\end{equation}
\end{ex}

In the examples \ref{ejs1} and \ref{ejs2} we define a binary
operation that makes an algebra a hemiimplicative semilattice. In
the rest of the paper we shall refer to this operation as the
implication of the algebra.

\begin{rem} \label{ejb}
The algebras with the implications (\ref{ej2}) and (\ref{ej5}) of
Example \ref{ejs2} satisfy $\mathrm{(H4)}$. The algebras with the
implication (\ref{ej6}) of Example \ref{ejs1} satisfy
$\mathrm{(H5)}$, and the algebras with the implication (\ref{ej1})
of Example \ref{ejs2} also satisfy $\mathrm{(H5)}$. For the
algebras with the implication (\ref{ej3}) of Example \ref{ejs1}
and by (\ref{ej4}) of Example \ref{ejs2}, if the universe of them
is not trivial, then $\mathrm{(H5)}$ is not satisfied because
$(0\we 1)\ra 1 = 0$. For the algebras with the implication
(\ref{ej6}) of Example \ref{ejs1} and with the implication
(\ref{ej1}) of Example \ref{ejs2} we have that $1\ra (1\we 0) = 0$
and $1\ra 1=1$, so $\mathrm{(H4)}$ is not satisfied by them.
\end{rem}

\begin{lem} \label{lm}
Let $A\in \WIS$ and $a,b\in A$. Then $a\rla b = b \rla a$, $a\rla
a = 1$ and $a\we (a\rla b) \leq b$.
\end{lem}

Inspired by Lemma \ref{lm} we introduce the following variety of
algebras.

\begin{defn}
We say that $(A,\s,\we,1)$ is a \emph{symmetric hemiimplicative
semilattice} if $(A,\we,\s,1)\in \WIS$ and $a\s b = b\s a$ for
every $a,b\in A$. We write $\gEA$ for the variety of symmetric
hemiimplicative semilattices.
\end{defn}

Let $A\in \gEA$. Define the binary operation $\Ra$ by $a\Ra b: = a
\s (a\we b)$.

\begin{lem}
Let $A\in \gEA$ and $a,b\in A$. Then
\begin{enumerate}[\normalfont 1)]
\item $a\Ra a = 1$, \item $a\we (a\Ra b) \leq b$, \item $a\Ra
(a\we b) = a\Ra b$, \item If $b\leq a$ then $a \RLa b = a\Ra b = a
\s b$.
\end{enumerate}
\end{lem}

\begin{proof}
Let $a,b\in A$. Then $a\Ra a = a\ra a = 1$, so $a\Ra a = 1$.
Besides $a\we (a\Ra b) = a \we (a\s (a\we b))\leq a\we b \leq b$.
By definition of $\Ra$ we have that $a\Ra (a\we b) = a\Ra b$.

Finally we shall prove that if $b\leq a$ then $a \RLa b = a\Ra b =
a \s b$. In order to show it, let $b\leq a$. We have that

\[
\begin{array}
[c]{lllll}
a\RLa b & = & (a\Ra b) \we (b\Ra a)&  & \\
 & = &  (a\s (a\we b)) \we (b \s (b\we a))&  & \\
 & = & (a\s b) \we (b\s b) &  &\\
 & = & (a\s b) \we 1 &  &\\
 & = & a\s b. &  &
\end{array}
\]

Moreover, $a\s b = a\s (a\we b) = a\Ra b$. Thus we obtain that if
$b\leq a$ then $a\RLa b = a\s b = a\Ra b$.
\end{proof}

In what follows we establish the relation between the varieties
$\WIS$ and $\gEA$.

\begin{prop}
\begin{enumerate}[\normalfont 1)]
\item If $(A,\ra, \we,1) \in \WIS$ then $(A,\we,\rla,1) \in \gEA$.
\item If $(A,\ra, \we,1) \in \WISWf$ then $a \ra b = a \Ra b$ for
every $a,b\in A$, where $\Ra$ is the implication associated to the
algebra $(A,\we,\rla,1)$. \item If $(A,\s,\we,1)\in \gEA$ then
$(A,\Ra,\we,1) \in \WISWf$.
\end{enumerate}
\end{prop}

We write $\gEAE$ for the subvariety of $\gEA$ whose algebras
satisfy the following condition:
\begin{enumerate}
\item[$\mathrm{(S)}$] $a\s b = (a\s (a\we b)) \we (b\s (a\we b))$.
\end{enumerate}

Equation $\mathrm{(S)}$ simply states that in $(A,\s,\we,1)\in
\gEAE$, we have that $a \RLa b = a\s b$ for every $a,b\in A$,
where $\Ra$ is the implication associated to the algebra
$(A,\we,\s,1)$.

\begin{cor} \label{termeq}
The varieties $\WISWf$ and $\gEAE$ are term equivalent.
\end{cor}

It is the case that $\gEAE$ is a proper subvariety of $\gEA$, as
the following example shows.

\begin{ex}
Let $A$ be the bounded lattice of the following figure:
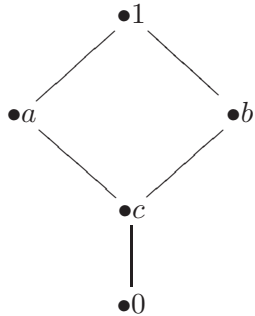
\[
\xymatrix{
   & \bullet 1 \ar@{-}\\
\ar@{-}[ur] \bullet a &  & \ar@{-}[ul] \bullet b\\
   & \ar@{-}[ul] \bullet c \ar@{-}[ur]&\\
   & \bullet 0 \ar@{-}[u]
}
\]
Define on $A$ the following binary operation:
 \[%
\begin{tabular}
[c]{c|ccccc}%
$\s$ & $0$ & $a$ & $b$ & $c$ & $1$\\\hline
$0$   & $1$ & $0$ & $0$ & $0$ & $0$\\
$a$   & $0$ & $1$ & $0$ & $c$ & $a$\\
$b$   & $0$ & $0$ & $1$ & $c$ & $b$\\
$c$   & $0$ & $c$ & $c$ & $1$ & $c$\\
$1$   & $0$ & $a$ & $b$ & $c$ & $1$%
\end{tabular}
\]

Straightforward computations show that $(A,\we,\s,1)\in \gEA$.
However, $a\s b = 0$, $a\s (a\we b) = c$ and $b\s (a\we b) = c$.
Thus, we obtain that $a\s b \neq (a\s (a\we b))\we (b\s (a\we
b))$. Therefore, $(A,\we,\s,1)\notin \gEAE$.
\end{ex}

\begin{rem}
For $i=1,\dots,6$, let $\mathrm{K}_i$ be the class of the algebras
in $\WIS$ where the implication is given by (i) of
examples \ref{ejs1} and \ref{ejs2}. It can be proved that these
classes are not closed by products; hence, they are not quasivarieties.
It would be interesting to
have an answer for the following general question: which is a set
of equations (quasiequations) for the variety (quasivariety)
generated by $\mathrm{K}_i$?
\end{rem}

\section{Congruences}
\label{section4}

In this section we study the congruences for some subclasses of
$\WIS$. More precisely, in Subsection \ref{sub1} we study the
lattice of congruences for the algebras given in examples
\ref{ejs1} and \ref{ejs2}. In Subsection \ref{sub2} we
characterize the principal congruences in $\WIS$.
\vspace{1pt}

In what follows we fix notation and we give some definitions. We
also give some known results about congruences on hemiimplicative
semilattices \cite{SM2}. Let $A\in \WIS$, $a,b \in A$ and $\theta$ a
congruence of $A$. We write $a/\theta$ to indicate the equivalence
class of $a$ associated to the congruence $\theta$ and
$\theta(a,b)$ for the congruence generated by the pair $(a,b)$.
Let $A\in \WIS$ or $A\in \gEA$. As usual, we say that $F$ is a
\emph{filter} if it is a subset of $A$ such satisfies the
following conditions: $1\in F$, if $a,b\in F$ then $a\we b \in F$,
if $a\in F$ and $a\leq b$ then $b\in F$. We also consider the
binary relation
\[
\Theta(F) = \{(a,b) \in A\times A: a\we f = y \we f\;\text{for
some}\; f\in F\}.
\]
Notice that if $A$ is a semilattice with greatest element (or a
bounded semilattice) and $F$ is a filter, then $\Theta(F)$ is the
congruence associated to the filter $F$. For $A\in \WIS$ and
$a,b,f\in A$ we define the following element of $A$: $t(a,b,f): =
(a \ra b) \rla ((a\we f) \ra (b\we f))$.

\begin{defn}
Let $A\in \WIS$ and $F$ a filter of $A$. We say that $F$ is a
\emph{congruent filter} if $t(a,b,f)\in F$ for every $a,b \in A$
and $f\in F$.
\end{defn}

The following result was proved in \cite{SM2}.

\begin{thm}
Let $A\in \WIS$. There exists an isomorphism between the lattice
of congruences of $A$ and the lattice of congruent filters of $A$,
which is established via the assignments $\theta \mapsto 1/\theta$
and $F\mapsto \Theta(F)$.
\end{thm}

Notice that if $(A,\we,\s,1)$ is a symmetric hemiimplicative
semilattice and $a,b,f \in F$ then $t(a,b,f) = (a\s b) \s ((a\we f) \s
(b\we f))$.

\begin{cor}
Let $A \in \gEA$. There exists an isomorphism between the lattice
of congruences of $A$ and the lattice of filters $F$ of $A$ which
satisfy $(a\s b) \s ((a\we f) \s (b\we f)) \in F$ for every
$a,b\in A$ and $f\in F$. The isomorphism is established via the
assignments $\theta \mapsto 1/\theta$ and $F\mapsto \Theta(F)$.
\end{cor}

\subsection{Congruences for some algebras of $\WIS$}\label{sub1}

In this subsection we characterize the congruences of examples
\ref{ejs1} and \ref{ejs2}. \vspace{1pt}

The following elemental remark will be used later.

\begin{rem} \label{rcf}
Let $A\in \WIS$ and $F$ a congruent filter of $A$. Then $F =
\{1\}$ if and only if $x/\Theta(F) = \{x\}$ for every $x$.
\end{rem}

\subsection*{Congruent filters associated to the implications
given by (\ref{ej3}) or (\ref{ej5})}

Let $F$ be a filter of $(A,\we,\ra,1) \in \WIS$, where $\ra$ is
the implication given by (\ref{ej3}) or (\ref{ej5}).

\begin{prop}\label{cf35}
$F$ is congruent if and only if $F = \{1\}$ or $F=A$.
\end{prop}

\begin{proof}
If the algebra $A$ is trivial the proof is immediate, so we can
assume that $A$ is an algebra not trivial. Suppose that $F$ is
congruent, and that $F \neq \{1\}$. By Remark \ref{rcf} we have
that there exist $x,y\in A$ such that $x\neq y$ and $x/\Theta(F) =
y/\Theta(F)$. Thus, there is $f$ in $F$ such that $x\we f = y\we
f$. Suppose that $x\nleq y$. Taking into account that $F$ is
congruent we have that $(x\ra y) \ra 1 \in F$. But $x\nleq y$, so
$x\ra y = 0$. Hence, $(x\ra y)\ra 1 = 0\ra 1$. However, $0\ra 1 =
0$ because $0\nleq 1$. Then we have that $0\in F$. Analogously we
can show that if $y\nleq x$ then $0\in F$ by considering $y\ra x$
in place of $x\ra y$. Therefore, $F = A$.
\end{proof}

\subsection*{Congruent filters associated to the implication
given by (\ref{ej6})}

Let $F$ be a filter of $(A,\we,\ra,1)$, where $\ra$ is the
implication given by (\ref{ej6}).

\begin{prop}
$F$ is congruent if and only if it satisfies the following
conditions for every $x,y\in A$ and $f\in F$:
\begin{enumerate}[\normalfont 1)]
\item If $x\nleq y$ and $x\we f \leq y \we f$ then $y\in F$. \item
If $x\nleq y$, $x\we f \nleq y\we f$ and $y\nleq f$ then $y \in
F$. \item If $x\leq y$ and $x\we f \nleq y\we f$ then $y \in F$.
\end{enumerate}
\end{prop}

\begin{proof}
Let $x,y\in A$ and $f\in F$. Let $x\nleq y$ and $x\we f\leq y\we
f$. As $t(x,y,f) = y \in F$ we have the condition 1). Suppose that
$x\nleq y$, $x\we f \leq y\we f$ and $y\nleq f$. Taking into
account that $t(x,y,f) = y\we f\in F$ we obtain the condition 2).
Finally suppose that $x\leq y$ and $x\we f\leq y\we f$. Since
$t(x,y,f) = y\we f\in F$ then we obtain 3).

Conversely, suppose that we have the conditions 1), 2) and 3). Let
$x,y\in A$ and $f\in F$. Suppose that $x\leq y$. If $x\we f\leq
y\we f$ then $t(x,y,f) = 1\in F$. If $x\we f\nleq y\we f$ then
$t(x,y,f) = y\we f$. By 3) we have that $y\in F$, so $t(x,y,f)\in
F$. Suppose now that $x\nleq y$. If $x\we f\nleq y\we f$ then
$t(x,y,f) = y \in F$ by 1). If $x\we f\nleq y\we f$ then $t(x,y,f)
= y \rla (y\we f)$. We have two possibilities: $y\leq f$ or
$y\nleq f$. If $y\leq f$ then $t(x,y,f) = 1\in F$. If $y\nleq f$
then $t(x,y,f) = y\we f$. By 2) we have that $y\in F$. Hence,
$t(x,y,f) \in F$. Therefore, $F$ is congruent.
\end{proof}

\subsection*{Congruent filters associated to the implication given
by (\ref{ej1})}

Let $F$ be a filter of $(A,\we,\ra,1)$, where $\ra$ is the
implication given by (\ref{ej1}).

\begin{prop}
$F$ is congruent if and only if it satisfies the following
conditions for every $x,y\in A$ and $f\in F$:
\begin{enumerate}[\normalfont 1)]
\item If $x\neq y$ and $x\we f = y \we f$ then $y\in F$. \item If
$x\neq y$, $x\we f \neq y\we f$ and $y\nleq f$ then $y \in F$.
\end{enumerate}
\end{prop}

\begin{proof}
Assume that $F$ is congruent, and let $x,y\in A$ such that $x\neq
y$. Let $f\in F$. Then $t(x,y,f) = y \rla ((x\we f)\ra (y\we f))$.
Moreover, $t(x,y,f)\in F$. If $x\we f = y \we f$ then $t(x,y,f) =
y \in F$. Suppose now that $x\we f \neq y\we f$ and $y\nleq f$
(i.e., $y\neq y \we f$). Thus, $t(x,y,f) = y \rla (y\we f) = y\we
f\in F$, so $y\in F$.

Conversely, suppose it holds the conditions 1) and 2). Consider
$x,y\in A$ and $f\in F$. If $x = y$ then $t(x,y,f) = 1\in F$. If
$x\neq y$ then $t(x,y,f) = y \rla ((x\we f)\ra (y \we f))$. If
$x\we f = y\we f$ then $t(x,y,f) = y$, which belongs to $F$ by
condition 1). Suppose that $x\we f \neq y\we f$. Then $t(x,y,f) =
y \rla (y\we f)$. If $y\leq f$ then $t(x,y,f) = 1\in F$. If
$y\nleq f$ then $t(x,y,f) = y\we f$, which also belongs to $F$ by
condition 2) (because $f$, $y\in F$).
\end{proof}

\subsection*{Congruent filters associated to the implication given
by (\ref{ej2})}

Let $F$ be a filter of $(A,\we,\ra,1)$, where $\ra$ is the
implication given by (\ref{ej2}).

\begin{prop}\label{cf2}
$F$ is congruent if and only if it satisfies the following
conditions for every $x,y\in A$ and $f\in F$:
\begin{enumerate}[\normalfont 1)]
\item If $x\nleq y$ and $x\we f \leq y \we f$ then $x\we y\in F$.
\item If $x\nleq y$, $x\we f \nleq y\we f$ and $x\we y\nleq f$
then $x\we y \in F$.
\end{enumerate}
\end{prop}

\begin{proof}
Suppose that $F$ is congruent. Let $x\nleq y$ and $x\we f \leq
y\we f$. Then $t(x,y,f) = x\we y \in F$, so we have the condition
1). Suppose that $x\nleq y$, $x\we f \nleq y\we f$ and $x\we y
\nleq f$. Then $t(x,y,f) = x\we y \we f \in F$, so $x\we y \in F$,
which is the condition 2).

Conversely, suppose that it holds the conditions 1) and 2). Let
$x,y \in A$ and $f\in F$. If $x\leq y$ then $t(x,y,f) = 1 \in F$.
Suppose that $x \nleq y$. If $x\we f \leq y\we f$ then $t(x,y,f) =
x\we y$, which belongs to $F$ by the condition 1). Suppose that
$x\nleq y$ and $x\we f \nleq y\we f$. Then $t(x,y,f) = (x\we y)
\rla (x \we y \we f)$. If $x\we y \leq f$ then $t(x,y,f) = 1\in
F$. If $x\we y \nleq f$ then $t(x,y,f) = x\we y \we f$. But by
condition 2) we have that $x\we y\in F$. Therefore, $t(x,y,f) \in
F$.
\end{proof}

\subsection*{Congruent filters associated to the implication
given by (\ref{ej4})}

Let $F$ be a filter of $(A,\we,\ra,1)$, where $\ra$ is the
implication given by (\ref{ej4}).

\begin{prop}\label{cf4}
$F$ is congruent if and only if it satisfies the following
conditions for every $x,y\in A$ and $f\in F$:
\begin{enumerate}[\normalfont 1)]
\item If $x\neq y$ and $x\we f = y \we f$ then $x\we y\in F$.
\item If $x\neq y$, $x\we f \neq y\we f$ and $x\we y\nleq f$ then
$x\we y \in F$.
\end{enumerate}
\end{prop}

\begin{proof}
Suppose that $F$ is congruent. In order to prove 1) and 2)
consider $x\neq y$. Suppose that $x\we f = y \we f$. Then we have
that $t(x,y,f) = (x\we y) \rla 1$. If $1\rla (x\we y) = 1$ then $1
= x\we y$, i.e., $x = y = 1$, which is an absurd. Then $t(x,y,f) =
x\we y \in F$. Hence we have the condition 1). Suppose now that
$x\we f \neq y\we f$ and $x\we y\nleq f$. Hence, $t(x,y,f) = (x\we
y) \rla (x\we y \we f)$. But $x\we y \nleq f$, so $t(x,y,f) = x\we
y \we f \in F$. Thus $x\we y\in F$, which is the condition 2).

Conversely, suppose that $F$ satisfies the conditions 1) and 2).
Let $x,y \in F$ and $f\in F$. If $x = y$ then $t(x,y,f) = 1\in F$.
Suppose that $x\neq y$, so $t(x,y,f) = (x\we y) \rla ((x\we f) \ra
(y\we f))$. If $x\we f = y\we f$ then $t(x,y,f) = x\we y$, which
belongs to $F$ by 1). Suppose that $x\we f \neq y\we f$, so
$t(x,y,f) = (x\we f) \rla (x\we y \we f)$. If $x\we y \leq f$ then
$t(x,y,f) = 1\in F$. If $x\we y\nleq f$ then by 2) we have that
$x\we y \in F$. But $t(x,y,f) = x\we y \we f$. Thus, $t(x,y,f)\in
F$.
\end{proof}

\subsection*{Totally ordered posets}

Let $A\in \WIS$ and $x,y,f\in A$. If $f\leq x\we y$ then $t(x,y,f)
= (x\ra y) \rla 1$, if $y\leq f\leq x$ then $t(x,y,f) = (x\ra y)
\rla (f\ra y)$, if $x\leq f \leq y$ then $t(x,y,f) = (x\ra y) \rla
(x\ra f)$ and if $x \leq f$ and $y\leq f$ then $t(x,y,f) = 1$.
Hence, we obtain the following result.

\begin{prop} \label{cadenas}
Let $A\in \WIS$ such that its underlying poset is a chain and let
$F$ be a filter of $A$. Then $F$ is congruent if and only if for
every $x,y \in A$ and $f\in F$ it holds the following conditions:
\begin{enumerate}[\normalfont (a)]
\item If $f\leq x\we y$ then $(x\ra y) \rla 1 \in F$. \item If
$y\leq f\leq x$ then $(x\ra y) \rla (f\ra y) \in F$. \item If
$x\leq f \leq y$ then $(x\ra y) \rla (x\ra f) \in F$.
\end{enumerate}
\end{prop}

\begin{cor}\label{ptop}
Consider the hemiimplicative semilattices, whose underlying order is total, with the implications
$\mathrm{(\ref{ej6})}$, $\mathrm{(\ref{ej1})}$,
$\mathrm{(\ref{ej2})}$ and $\mathrm{(\ref{ej4})}$ (see examples
\ref{ejs1} and \ref{ejs2}). Every filter in these algebras is congruent.
\end{cor}

\begin{proof}
Let $x,y \in A$ and $f\in F$. If $f\leq x\we y$. Then $(x\ra y)
\rla 1 \in \{x\we y,y,1\} \subseteq F$. If $y\leq f\leq x$ then
$(x\ra y) \rla (f\ra y) \in \{f,1\} \subseteq F$. If $x\leq f \leq
y$ then $(x\ra y) \rla (x\ra f) \in \{f,1\}\subseteq F$.
Therefore, it follows from Proposition \ref{cadenas} that $F$ is
congruent.
\end{proof}

Notice, however, that it follows from Proposition \ref{cf35} that there
are examples of hemiimplicative semilattices where the order is
total with the property that not every filter is congruent.

On the other hand, it is not the case that every filter in a non
totally ordered algebra of the classes considered in previous
corollary is congruent. Consider, for example, the boolean lattice
of four elements, where $x$ and $y$ are the atoms, and let
$F = \{x,1\}$. We write $B$ for
the universe of this algebra. We have that $F$ is a filter of
$(B,\we, \ra_i, 1)$, where $\ra_i$ is the implication given in (i)
of the above mentioned examples, for $i=1, \dots,6$. We write
$t_i$ for the ternary term $t$ over the algebra $(B,\we, \ra_i,
1)$. Since $t_1(x,y,x)$, $t_2(y,x,x)$, $t_3(0,y,x)$, $t_4(0,y,x)$,
$t_5(y,x,x)$ and $t_6(x,y,x)$ are the bottom then $F$ is not a
congruent filter.

\subsection{Principal congruences for algebras of $\WIS$}\label{sub2}

Let $A\in \WIS$ and $a,b\in A$. We write $F^{c}(a)$ for the
congruent filter generated by $\{a\}$. In \cite{SM2} it was proved
that if $\theta(a,b)$ is the congruence generated by $(a,b)$, then
$(x,y) \in \theta(a,b)$ if and only if $x\rla y \in F^{c}(a\rla
b)$. We will give some definitions in order to make possible an
explicit description of $F^{c}(a)$. \vspace{1pt}

For $X\subseteq A$ we define $t(x,y,X) = \{t(x,y,z):z\in X\}$.
Then we define $t^{+}(x,y,X)$ as the elements $z\in A$ such that
$z\geq t(x,y,w_1)\we \dots \we t(x,y,w_k) \we w_{k+1}\we \dots
w_{k+t}$ for some $w_i\in X$. In a next step we define
\[
t(X) = \bigcup_{x,y\in A}t^{+}(x,y,X).
\]
We also define $T_{0}(X) = X$, $T_{n+1}(X) = t(T_{n}(X))$ and
$T(X) = \bigcup_{n \in \mathbb{N}} T_n (X)$, where $\mathbb{N}$ is
the set of natural numbers. It is immediate that $a\in T(\{a\})$.

\begin{prop}
Let $A\in \WIS$ and $a\in A$. Then $F^{c}(a) = T(\{a\})$.
\end{prop}

\begin{proof}
Straightforward computations based on induction show that
\begin{equation}\label{ein}
T_{n}(\{a\}) \subseteq T_{n+1}(\{a\})
\end{equation}
for every $n$. We use this property throughout this proof.

We have that $1\in T(\{a\})$. It follows from the construction
that $T(\{a\})$ is an upset. In order to prove that this set is
closed by $\we$, let $z,z' \in T(\{a\})$. Then there are $n$ and
$m$ such that $z\in T_{n}(\{a\})$ and $z'\in T_{m}(\{a\})$. By
$(\ref{ein})$ we have that $z,z'\in T_{p}(\{a\})$, where $p$ is
the maximum between $n$ and $m$. Straightforward computations
prove that $z\we z'\in T_{p}(\{a\})$, so $z\we z' \in T(\{a\})$.
Hence, $T(\{a\})$ is a filter.

In order to show that $T(\{a\})$ is congruent, let $z\in T(\{a\})$
and $x,y \in A$. Then, there is $n$ such that $z\in T_{n}(\{a\})$.
Taking into account that $t(x,y,z)\geq t(x,y,z) \we z$, we have
that $t(x,y,z) \in t(T_{n}(\{a\}))$, i.e., $t(x,y,z) \in
T_{n+1}(\{a\}) \subseteq T(\{a\})$. Thus, $T(\{a\})$ is congruent.

Finally we show that $F^{c}(a) = T(\{a\})$. Let $F$ be a congruent
filter such that $a\in F$. We need to prove that
$T(\{a\})\subseteq F$, i.e., that $T_{n}(\{a\})\subseteq T(\{a\})$
for every $n$. If $n = 0$ is immediate. Suppose that
$T_{n}(\{a\})\subseteq T(\{a\})$ for some $n$. We shall prove that
$T_{n+1}(\{a\})\subseteq T(\{a\})$. Let $z\in T_{n+1}(\{a\})$.
Then there are $x,y\in A$ and $w_1,...,w_{k+t} \in T_{n}(\{a\})$
such that
\[
z\geq t(x,y,w_1)\we...t(x,y,w_k)\we w_{k+1}\we...\we w_{k+t}.
\]
But $T_{n}(\{a\})\subseteq F$, and $F$ is a congruent filter.
Thus,
\[
t(x,y,w_1)\we...t(x,y,w_k)\we w_{k+1}\we...\we w_{k+t} \in F.
\]
Hence, $z\in F$. Therefore, $T_{n+1}(\{a\}) \subseteq F$, which
was our aim.
\end{proof}

\begin{cor}
Let $A\in \WIS$ and $a,b\in A$. Then $(x,y)\in \theta(a,b)$ if and
only if $x \rla y\in T(\{a\rla b\})$.
\end{cor}

\subsection*{Acknowledgments}

This work was supported by CONICET, Argentina [PIP
112-201101-00636] and Universidad Nacional de La Plata [11/X667].


\noindent Jos\'e Luis Castiglioni,\\
Departamento de Matem\'atica, \\
Facultad de Ciencias Exactas (UNLP), \\
and CONICET.\\
Casilla de correos 172,\\
La Plata (1900),\\
Argentina.\\
jlc@mate.unlp.edu.ar\\

\noindent Hern\'an Javier San Mart\'in,\\
Departamento de Matem\'atica, \\
Facultad de Ciencias Exactas (UNLP), \\
and CONICET.\\
Casilla de correos 172,\\
La Plata (1900),\\
Argentina.\\
hsanmartin@mate.unlp.edu.ar

\end{document}